\newif\ifdraft
\draftfalse

\documentclass[reqno,a4paper]{amsart}

\usepackage{mathrsfs,amsmath,amssymb,graphicx}
\usepackage{stmaryrd}
\usepackage[unicode]{hyperref}
\usepackage{latexsym}
\usepackage{layout}
\usepackage[english]{babel}
\usepackage{color}
\usepackage{subcaption}
\usepackage{fancyvrb}
\usepackage{color}
\usepackage{amsmath,amscd}

\usepackage{multirow}

\usepackage{mathtools}
\usepackage{tikz} 
 
\usetikzlibrary{matrix,arrows,decorations,shapes,calc}
\tikzset{every picture/.append style={remember picture},
na/.style={baseline=-.5ex}}

\theoremstyle{plain}
\newtheorem{theorem}{Theorem}[section]
\newtheorem{lemma}[theorem]{Lemma}

\newtheorem{corollary}[theorem]{Corollary}
\newtheorem{conjecture}[theorem]{Conjecture}
\newtheorem{definition}{Definition}[section]
\theoremstyle{remark}

\newcommand{\nada}[1]   {}

\newcommand{\onesum} {-{}-}

\newcommand{\Stwo} {\mathfrak S}

\newcommand{\Sbb} {\mathbb S}

\newcommand{\sphere} {\Sbb^3}

\definecolor{mygray}{rgb}{0.92,0.92,0.92}

\newcommand{\Compl}[1]{\overline{\sphere\setminus #1}}

\newcommand{\op}[1]{\operatorname{#1}}

\graphicspath{{figures/}}

\numberwithin{equation}{section}
\numberwithin{figure}{section}

\title{Numerical irreducibility criteria for handlebody links} 
\author{Giovanni Bellettini}
\address{Dipartimento di Ingegneria dell'Informazione e Scienze Matematiche, Universit\`a di Siena, 53100 Siena, Italy,
and International Centre for Theoretical Physics ICTP,
Mathematics Section, 34151 Trieste, Italy}
\email{bellettini@diism.unisi.it}
\author{Maurizio Paolini}
\address{Dipartimento di Matematica e Fisica, Universit\`a Cattolica del Sacro Cuore, 25121 Brescia, Italy}
\email{maurizio.paolini@unicatt.it}
\author{Yi-Sheng Wang}
\address{National Center for Theoretical Sciences, Mathematics Division, Taipei 106, Taiwan}
\email{yisheng@ncts.ntu.edu.tw}


\begin{document}

\thanks{}

\begin{abstract}
In this paper we define a set of numerical criteria for a handlebody link
to be irreducible. It provides an effective, easy-to-implement method to determine the irreducibility
of handlebody links; particularly, it recognizes
the irreducibility of all 
handlebody knots in the Ishii-Kishimoto-Moriuchi-Suzuki knot table  
and most handlebody links in the Bellettini-Paolini-Paolini-Wang link table. 
\end{abstract}

\maketitle

%
%
%

\section{Introduction}\label{sec:intro}

A handlebody link $\op{HL}$ 
is a union of finitely many handlebodies of positive genus
embedded in the $3$-sphere $\sphere$; 
two handlebody links
are equivalent if they are ambient isotopic \cite{Suz:70}, 
\cite{Ish:08}. Throughout the paper  
handlebody links are non-split  
unless otherwise specified. 
 
A handlebody link $\op{HL}$ is reducible if there exists 
a cutting $2$-sphere $\Stwo$ in $\sphere$ such that $\Stwo$ and
$\op{HL}$ intersect transversally at 
an incompressible disk $D$ 
in $\op{HL}$ 
(Fig.\ \ref{intro:fig:reducible_hl}); otherwise
it is irreducible. 
Note that a cutting sphere
\begin{figure}[ht]
\def\svgwidth{0.9\columnwidth}
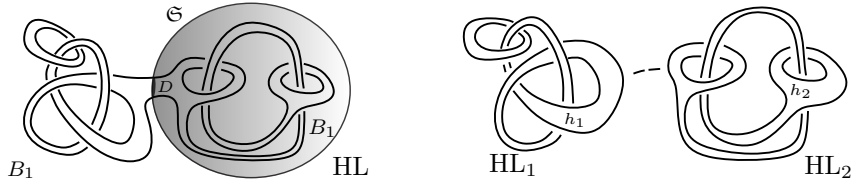   
\caption{A reducible handlebody link and its factors}
\label{intro:fig:reducible_hl}
\end{figure}
$\Stwo$ of a reducible handlebody link $\op{HL}$
factorizes it into 
two handlebody links $\op{HL}_1,\op{HL}_2$,
where $\op{HL}_i:=\op{HL}\cap B_i$, and $B_i$, $i=1,2$, are
the closures of components of the 
complement $\Compl{\Stwo}$ (Fig.\ \ref{intro:fig:reducible_hl});
the factorization is denoted by 
\begin{equation}\label{intro:eq:one_sum_factorization}
\op{HL}=(\op{HL}_1,h_1)\text{\onesum} (\op{HL}_2,h_2),
\end{equation}
and we call $\op{HL}_i$, $i=1,2$, a factor of the factorization,
where $h_1,h_2$ are components of $\op{HL}_1,\op{HL}_2$
containing $D$, respectively.
%

Handlebody links are often studied and visualized via 
diagrams of their spines \cite{Ish:08};
it is, however, not an easy task to detect 
the irreducibility of a handlebody link
from its diagram.  
The complexity lies in the IH-move 
\cite{Ish:08}.
In fact, it is not known whether we have an affirmative answer 
to Conjecture \ref{intro:conj:minimal_diag} or 
Conjecture \ref{intro:conj:additivity}\footnote{Conj.\ \ref{intro:conj:minimal_diag} implies 
Conj. \ref{intro:conj:additivity} in some special cases \cite[Theorem $2$]{Tsu:70} and \cite[Theorem $6.1$]{BePaPaWa:20}.}.
\begin{conjecture}\label{intro:conj:minimal_diag} 
Every reducible handlebody link
admits a minimal diagram whose underlying plane graph is 
$1$-edge-connected. 
\end{conjecture}
\begin{conjecture}\label{intro:conj:additivity}
The crossing number of a reducible handlebody link
is the sum of crossing numbers of its factors:  
\begin{equation}\label{eq:c_plus_c}
c\big((\op{HL}_1,h_1)\text{\onesum} (\op{HL}_2,h_2)\big)=c(\op{HL}_1)+c(\op{HL}_2).
\end{equation}
\end{conjecture}
If either conjecture is true, it implies
the reducible handlebody link table \cite[Table $5$]{BePaPaWa:20}
is complete, and thus,  
the irreducibility of all handlebody links 
in \cite[Table $1$]{BePaPaWa:20} but $6_9$
can be proved by simply
comparing their \emph{$ks_{G}$-invariants} \cite{KitSuz:12}.
The invariant $ks_{G}(\op{HL})$ is 
the number of conjugacy classes of homomorphisms 
from the knot group $G_{\op{HL}}$, 
the fundamental group of $\op{HL}$'s complement, to 
a finite group $G$---two homomorphisms are  
in the same conjugacy class if they
are conjugate.

We do not purse these conjectures here but
instead introduce some numerical criteria 
for a handlebody link to be irreducibile. 
Other irreducibility tests using quandle invariants have been developed
by Ishii and Kishmoto \cite{IshKis:11}, and are 
used in the classification of irreducible handlebody knots
of genus $2$ \cite{IshKisMorSuz:12}. 
 
\noindent 
\textbf{Main Results \& Structure.}
A handlebody link $\op{HL}$ is said to be of type $[n_1,n_2,...n_m]$
if it consists of $n_i$ handlebodies of genus $i$, $i=1,\dots,m$,
and a handlebody link is $r$-generator if its knot group 
is of rank $r$. Note that $r$ is necessarily larger than or equal to the 
genus $g(\op{HL})$ of $\op{HL}$, which is the sum $\sum_{i=1}^{m}i\cdot n_i$
of genera of components of $\op{HL}$.
Let $A_4, A_5$ be alternating groups of degree $4,5$, respectively.
\begin{theorem}[\textbf{Necessary conditions for reducibility--$A_4$}]\label{intro:thm:A4_criteria}
Let $\op{HL}$ be a reducible handlebody link of genus $g$.
If the trivial knot is a factor of some factorization of $\op{HL}$,
then  
\begin{equation}\label{intro:eq:criterion_trivial_knot}
12  \mid \op{ks}_{A_4}(\op{HL})+6\cdot 3^{g-1}+2\cdot 4^{g-1}; 
\end{equation}
if a $2$-generator knot is a factor of some factorization of $\op{HL}$, 
then 
\begin{equation}\label{intro:eq:criterion_2-gen_knot}
12+24k  \mid \op{ks}_{A_4}(\op{HL})+(6+16k)\cdot 3^{g-1}+(2+6k)\cdot 4^{g-1}, k=0\text{ or }1;
\end{equation}
if a $2$-generator link is a factor of some factorization of $\op{HL}$, 
then
\begin{equation}\label{intro:eq:criterion_2-gen_link}
48+24k \mid ks_{A_4}(\op{HL})+(26+16k)\cdot 3^{g-1} + (8+6k)\cdot 4^{g-1}, k=0,1,2,3\text{ or } 4.
\end{equation} 
\end{theorem}
 
\begin{theorem}[\textbf{Necessary conditions for reducibility--$A_5$}]\label{intro:thm:A5_criteria}
Let $\op{HL}$ be a reducible handlebody link of genus $g$.  
If the trivial knot is a factor of some factorization of $\op{HL}$, 
then 
\begin{equation}\label{intro:eq:criterion_trivial_knot_A5}
60\mid \op{ks}_{A_5}(\op{HL})+14\cdot 4^{g-1}+19\cdot 3^{g-1}+22\cdot 5^{g-1}.
\end{equation}
\end{theorem}  
 
From these necessary conditions 
we derive the irreducibility test for handlebody knots
of genus up to $3$ and handlebody links of various types.
\begin{corollary}\label{intro:cor:A4-A5_criteria_HK}
Given a $r$-generator handlebody knot $\op{HL}$ of genus $g$,
if $r=g+1$ and $\op{HL}$ fails to satisfy either \eqref{intro:eq:criterion_trivial_knot} or \eqref{intro:eq:criterion_trivial_knot_A5},
then $\op{HL}$ is irreducible;
if $r=g+2$ and $\op{HL}$ fails to satisfy both
\eqref{intro:eq:criterion_trivial_knot} and 
\eqref{intro:eq:criterion_2-gen_knot},
then $\op{HL}$ is irreducible. 
\end{corollary}

The situation with multi-component handlebody links is 
slightly more complicated as there are more possible combinations;
thus we summarize it in a tabular format in Table \ref{intro:tab:applications_links}, which is also
a corollary of Theorems \ref{intro:thm:A4_criteria} and \ref{intro:thm:A5_criteria}.  
The left two columns in Table \ref{intro:tab:applications_links}
list criteria which if a handlebody link fails, it is irreducible.
Be aware ``$\&$ (i.e. and)'' and ``\emph{or}'' in those two columns.

\begin{table}[ht]  
\caption{Tests for irreducibility of handlebody links (more than one component)}
\begin{tabular}{c|c|c|c}
\multirow{2}{*}{no. of components}&\multirow{2}{*}{type} & $r=g$ & $r=g+1$\\
\cline{3-4}
&&\multicolumn{2}{c}{$\op{HL}$ is irreducible if it fails criterion/criteria}\\
     \hline
    \multirow{4}{*}{$2$}
& $[1,1]$ &    \eqref{intro:eq:criterion_trivial_knot} or \eqref{intro:eq:criterion_trivial_knot_A5}
   &   \eqref{intro:eq:criterion_trivial_knot} \&  \eqref{intro:eq:criterion_2-gen_knot}\\ \cline{2-4}
   &$[0,2]$ &  \eqref{intro:eq:criterion_trivial_knot} or \eqref{intro:eq:criterion_trivial_knot_A5}
   &   \eqref{intro:eq:criterion_trivial_knot} \& \eqref{intro:eq:criterion_2-gen_knot}\\ \cline{2-4}
      &$[1,0,1]$ &    \eqref{intro:eq:criterion_trivial_knot} or \eqref{intro:eq:criterion_trivial_knot_A5}
   &   \eqref{intro:eq:criterion_trivial_knot}, \eqref{intro:eq:criterion_2-gen_knot} \& \eqref{intro:eq:criterion_2-gen_link}\\ \cline{2-4}
      &$[0,1,1]$ &    \eqref{intro:eq:criterion_trivial_knot} or \eqref{intro:eq:criterion_trivial_knot_A5}
   &  not applicable \\
     \hline
%
\multirow{3}{*}{$3$}
& $[2,1]$ &    \eqref{intro:eq:criterion_trivial_knot} \& \eqref{intro:eq:criterion_2-gen_link}
   & \eqref{intro:eq:criterion_trivial_knot}, \eqref{intro:eq:criterion_2-gen_knot} \& \eqref{intro:eq:criterion_2-gen_link} \\ \cline{2-4}
   &$[1,2]$ &  \eqref{intro:eq:criterion_trivial_knot} \& \eqref{intro:eq:criterion_2-gen_link}
   &   \multirow{2}{*}{not applicable}\\ \cline{2-3}
   
   &$[2,0,1]$ &   \eqref{intro:eq:criterion_trivial_knot} \& \eqref{intro:eq:criterion_2-gen_link}
   &  \\
     \hline      
4& $[3,1]$ &   \eqref{intro:eq:criterion_trivial_knot} 
\& \eqref{intro:eq:criterion_2-gen_link}&
not applicable\\ \hline    
     
\end{tabular}
\label{intro:tab:applications_links}
\end{table}

The set of irreducibility criteria is put to test 
in Section \ref{sec:examples};
it detects the irreducibility of 
all handlebody knots, which are of type $[0,1]$, in 
the Ishii-Kishimoto-Moruichi-Suzuki knot table \cite{IshKisMorSuz:12}
and the irreducibility of 
all handlebody links, which are of type $[1,1],[2,1]$ or $[3,1]$, 
but two ($6_9,6_{12}$),  
in the Bellettini-Paolini-Paolini-Wang link table
\cite{BePaPaWa:20}, showing that it is highly sensitive to 
the irreducibility of a handlebody link.

The major constraint of the irreducibility test  
is that the rank of the knot group $G_{\op{HL}}$ 
cannot be too large and the difference between
the rank and the genus $g(\op{HL})$ needs to be small;
on the other hand, the criteria are easy to implement 
and can be computed by a code.

The paper is organized as follows: Section \ref{sec:prelim} recalls 
basic properties of handlebody links and knot groups. 
The necessary conditions for reducibility 
(Theorems \ref{intro:thm:A4_criteria} and 
\ref{intro:thm:A5_criteria})
are proved in Section \ref{sec:test}.
Section \ref{sec:examples} records results of 
the irreducibility test applying to various families of handlebody links.
Lastly, the existence of irreducible handlebody links 
of any given type is proved by a concrete construction 
making use of a generalized knot sum for
handlebody links.


\section{Preliminaries}\label{sec:prelim}
Throughout the paper we work in the piecewise linear category.
We use 
$\op{HL}$ to refer to general handlebody links (including handlebody knots),
and use $\op{HK}, K$ or $L$ when referring specifically for handlebody knots,
knots or links, respectively. 
$G_{\bullet}$ denotes the knot group
of $\bullet=\op{HL},\op{HK},K$ or $L$; $\simeq$
stands for an isomorphism of groups.
To begin with, we review some basic properties of 
reducible handlebody links and the free product of groups.
 

\begin{definition}
The rank $rk(G)$ of a finitely generated group $G$ 
is the smallest cardinality of a generating set of $G$.
\end{definition}
\begin{definition}
A handlebody link is $r$-generator if its knot group is of rank $r$.
\end{definition}

The rank respects the free product of groups \cite{Gru:40}. 
\begin{lemma}[Grushko theorem]\label{lm:Gru_theorem}
If $G=G_1\ast G_2$, then 
\[rk(G)=rk(G_1)+rk(G_2).\]
\end{lemma}

\begin{lemma}\label{lm:rank=deficiency}
A $g$-generator handlebody knot $\op{HK}$ of genus $g$ is trivial.
\end{lemma}
\begin{proof}
By the exact sequence of group homology \cite{Sta:65b},
the deficiency $d$ of the knot group of $\op{HK}$ is at most $g$;
on the other hand, the Wirtinger presentation induces 
a presentation with deficiency $g$, so we have $d=g$.
By \cite[Satz $1$]{Mag:39}, \cite{Stam:67},  
the knot group is free, and therefore $\op{HK}$ is trivial. 
\end{proof}

The following are corollaries of Lemmas \ref{lm:Gru_theorem} and \ref{lm:rank=deficiency} and the fact that  
$\op{HL}=(\op{HL}_1,h_1)\text{\onesum}(\op{HL}_2,h_2)$ implies 
then $g(\op{HL})=g(\op{HL}_1)+g(\op{HL}_2)$. 
The corollaries, together with Theorems \ref{intro:thm:A4_criteria}
and \ref{intro:thm:A5_criteria}, give Corollary \ref{intro:cor:A4-A5_criteria_HK} and Table \ref{intro:tab:applications_links}.

\begin{corollary}\label{cor:red_handlebody_knot}
A $(g+1)$-generator handlebody knot $\op{HK}$ of genus $g=2,3$ 
is reducible if and only if the trivial knot is a factor 
of some factorization of $\op{HK}$.
\end{corollary}

\begin{corollary}\label{cor:red_2-handlebody_link_r=g}
A $2$-component, $g$-generator 
handlebody link $\op{HL}$ of genus $g\leq 5$
is reducible if and only if 
the trivial knot is a factor of some factorization of $\op{HL}$.
\end{corollary}

\begin{corollary}\label{cor:red_2-handlebody_link_r=g+1}
A genus $g$, $(g+1)$-generator 
handlebody link $\op{HL}$ of type $[1,1]$ or $[0,2]$ 
is reducible if and only if 
the trivial knot or a $2$-generator knot is
a factor of some factorization of $\op{HL}$.
\end{corollary}

\begin{corollary}\label{cor:red_m-handlebody_link_r=g}
A $3$- or $4$-component, $g$-generator 
handlebody link $\op{HL}$ of genus $g\leq 5$ 
is reducible if and only if 
the trivial knot or a $2$-generator link is a factor of 
some factorization of $\op{HL}$.
\end{corollary}

\begin{corollary}\label{cor:red_m-handlebody_link_g=5}
A $5$-generator handlebody link $\op{HL}$ of type $[1,0,1]$ or $[2,1]$ 
is reducible if and only if 
the trivial knot, $2$-generator knot, or $2$-generator link  
is a factor of some factorization of $\op{HL}$.
\end{corollary}



\section{Irreducibility tests}\label{sec:test}
\subsection{Homomorphisms to a finite group}
\begin{definition} 
Given a handlebody link $\op{HL}$ and a finite group $G$,
$ks_{G}(\op{HL})$ is the number of conjugacy 
classes of homomorphissm 
from $G_{\op{HL}}$ to $G$, $ks_{H}^{G}(\op{HL})$ 
is the number  
of conjugacy classes of homomorphisms 
from $G_{\op{HL}}$ to a subgroup of $G$
isomorphic to $H$, and 
$ks_{G}^{w}(\op{HL})$ is the number of 
homomorphisms from $G_{\op{HL}}$ to $G$. 
\end{definition}

\begin{lemma}\label{lm:formula_of_ks_G}
Suppose any subgroup of $G$ 
either has trivial centralizer or is abelian, and any 
two maximal abelian subgroups of $G$ have trivial intersection. Let
$H_i$, $i=1,\dots,n$, be isomorphism types of 
maximal abelian subgroups of $G$, and $l_i$ be
the number of maximum abelian subgroups isomorphic to $H_i$.
Then for any handlebody link $\op{HL}$, 
$ks_{G}(\op{HL})$ can be expressed
in terms of $ks_{G}^w(\op{HL})$ and $ks_{H_i}^{G}(\op{HL})$   
\begin{multline}\label{eq:formula_ks_G_via_ks_HG_ks_Hw}
ks_{G}(\op{HL})=ks_{H_1}^{G}(\op{HL})
+\cdots+ks_{H_n}^{G}(\op{HL})
-n+1\\
+\frac{ks_{G}^w(\op{HL})
-l_1(ks_{H_1}^{w}(\op{HL})-1)
-\cdots -l_n(ks_{H_n}^w(\op{HL})-1) 
-1}{\vert G\vert}. 
\end{multline}  
\end{lemma}
\begin{proof}
The difference
\begin{equation}\label{eq:conjugacy_homo_centerless}
ks_{G}(\op{HL})-\left( ks_{H_1}^{G}(\op{HL})
+\cdots+ks_{H_n}^{G}(\op{HL})
-n+1\right)
\end{equation}
is the number of conjugacy classes of 
homomorphisms $G_{\op{HL}}\rightarrow G$ 
whose images have trivial centralizers.
On the other hand, for such a homomorphism $\phi$, we have 
\[\phi\neq g\cdot\phi\cdot g^{-1},\]
for any non-trivial element $g\in G$, and hence 
the conjugacy class of $\phi$
contains $\vert G\vert$ members.
Now, since
the intersection of any two maximal abelian subgroups 
is trivial, the difference
\begin{equation}\label{eq:homo_centerless}
ks_{G}^w(\op{HL})
-l_1(ks_{H_1}^{w}(\op{HL})-1)
-\cdots -l_n(ks_{H_n}^w(\op{HL})-1) 
-1
\end{equation}
is the number of homomorphisms $G_{\op{HL}}\rightarrow G$
whose images have trivial centralizers. 
Therefore dividing \eqref{eq:homo_centerless} by $\vert G\vert$ 
gives us \eqref{eq:conjugacy_homo_centerless}, that is,
\begin{multline*}
\frac{ks_{G}^w(\op{HL})
-l_1(ks_{H_1}^{w}(\op{HL})-1)
-\cdots -l_n(ks_{H_n}^w(\op{HL})-1) 
-1}{\vert G\vert}\\
=ks_{G}(\op{HL})-\left( ks_{H_1}^{G}(\op{HL})
+\cdots+ks_{H_n}^{G}(\op{HL})
-n+1\right),
\end{multline*} 
and this proves the formula \eqref{eq:formula_ks_G_via_ks_HG_ks_Hw}.
\end{proof}
It is not difficult to check that $A_4,A_5$ satisfy conditions 
in Lemma \ref{lm:formula_of_ks_G}, whence 
we derive the following formulas.   
\begin{corollary}\label{lm:formualas_of_ks_A4_ks_A5}
Let $\mathbb{Z}_n$ be the cyclic group of order $n$, and $V_4\simeq \mathbb{Z}_2\oplus \mathbb{Z}_2$. Then
\begin{align}
ks_{A_4}(\op{HL})&=ks_{V_4}^{A_4}(\op{HL})
+ ks_{\mathbb{Z}_3}^{A_4}(\op{HL})-1
+\frac{ks_{A_4}^w(\op{HL})
-4(ks_{\mathbb{Z}_3}^{w}(\op{HL})-1)
-ks_{V_4}^w(\op{HL})}{12}\label{eq:formula_ks_ksw_A4}\\
ks_{A_5}(\op{HL})&=ks_{V_4}^{A_5}(\op{HL})
+ks_{\mathbb{Z}_3}^{A_5}(\op{HL})
+ks_{\mathbb{Z}_5}^{A_5}(\op{HL})
-2\label{eq:formula_ks_ksw_A5}\\
+&\frac{ks_{A_5}^w(\op{HL})
-10(ks_{\mathbb{Z}_3}^{w}(\op{HL})-1)
-5(ks_{V_4}^w(\op{HL})-1)
-6(ks_{\mathbb{Z}_5}^w(\op{HL})-1)
-1}{60}.\nonumber
\end{align} 
\end{corollary}

Given an injective homomorphism $H\xrightarrow{\iota} G$, 
then the number $n_H$ of conjugacy classes of elements in $G$
representable by elements in $\iota(H)$ is 
independent of $\iota$ if any two 
subgroups of $G$ isomorphic to $H$ are conjugate.
If furthermore $\iota(H)$ is a maximal abelian
subgroup with $\iota(H)$ being the centralizer of every
element in $\iota(H)$, then
$ks_{H}^{w}(\op{HL}), ks_{H}^{G}(\op{HL})$
can be computed explicitly.  

\begin{lemma}\label{lm:formulas_ks_H}
Under the assumptions preceding the lemma, if $g(\op{HL})=g$, then 
\[
ks_{H}^w(\op{HL})=\vert H\vert^g\quad\textbf{and}\quad ks_H^G(\op{HL})
=(n_{H}-1)\cdot\frac{\vert H\vert^{g}-\vert H\vert}{\vert H\vert-1}+n_H.
\]
\end{lemma}
\begin{proof}
Firstly, since $H$ is abelian, any homomorphism
from $G_{\op{HL}}$ to $H$ factors through the abelianization of $G_{\op{HL}}$,
which is the free abelian group $\mathbb{Z}^{g}$ of rank $g$. 
Especially, $ks_{H}^w(\op{HL})$ (resp. $ks_H^G(\op{HL})$)
is equal to the numbers (resp.\ of conjugacy classes) of homomorphisms
from $\mathbb{Z}^g$ to $H$. This implies the first identity.

For the second identity, we let
\[ks_H^G(\op{HL})=l_g\] 
and $\op{id},h_2,\dots,h_{n_H}\in \iota(H)<G$  
be selected representatives of 
the $n_H$ conjugacy classes of elements in $G$.
Note that if $g=1$, we have $l_1=n_H$.

For $g>1$, up to conjugation, we may assume
the $g$-th copy of $\mathbb{Z}^g$
is sent to $h\in\{\op{id},h_2\dots,h_{n_H}\}$. 
There are $l_{g-1}$ homomorphisms when $h=\op{id}$, 
and $\vert H\vert^{g-1}$ homomorphisms 
when $h=h_i,i=2,\dots,n_H$, 
because the centralizer of $h_i$ is $\iota(H)$. 
As a result, we obtain the recursive formula
\[l_g=l_{g-1}+(n_{H}-1)\cdot\vert  H\vert^{g-1},\]
and hence
\begin{equation}\label{eq:sum_of_differences}
l_g-l_1=\sum_{k=2}^{g}(l_k-l_{k-1})= \sum_{k=2}^{g}(n_H-1)\cdot\vert H\vert^{k-1}
=(n_H-1)\cdot\frac{\vert H\vert^{g}-\vert H\vert}{\vert H\vert-1}.
\end{equation} 
This implies the second equality after we substitute $l_1=n_H$
into \eqref{eq:sum_of_differences}.
\end{proof}

Maximal abelian subgroups of $A_4,A_5$ satisfy
conditions assumed in Lemma \ref{lm:formulas_ks_H}, and hence  
we have the formulas:
\begin{align}
ks_{\mathbb{Z}_3}^w(\op{HL}) &=3^g; \quad   
ks_{V_4}^w(\op{HL}) =4^g; \quad     
ks_{\mathbb{Z}_5}^w(\op{HL}) =5^g,\label{eq:formulas_ksw}\\
ks_{\mathbb{Z}_3}^{A_4}(\op{HL})&=3^g; \quad
 ks_{V_4}^{A_4}(\op{HL})=\frac{4^g-4}{3}+2,\label{eq:formulas_ksA4}\\
ks_{\mathbb{Z}_3}^{A_5}(\op{HL}) 
 &=\frac{3^g-3}{2}+2; \quad  
 ks_{V_4}^{A_5}(\op{HL})
 =\frac{4^g-4}{3}+2,\quad  
 ks_{\mathbb{Z}_5}^{A_5}(\op{HL}) 
 =\frac{5^g-5}{2}+3,\label{eq:formulas_ksA5}
\end{align} 
  
Plugging \eqref{eq:formulas_ksw},  \eqref{eq:formulas_ksA4}
into \eqref{eq:formula_ks_ksw_A4}, and 
\eqref{eq:formulas_ksw}, \eqref{eq:formulas_ksA5} into \eqref{eq:formula_ks_ksw_A5}
gives the following:
\begin{corollary}\label{cor:formuala_ks_ksw}
For a genus $g$ handlebody link $\op{HL}$, we have
\begin{align*}
ks_{A_4}^w(\op{HL})&=12ks_{A_4}(\op{HL})-8\cdot 3^{g}-3\cdot 4^g\\
ks_{A_5}^w(\op{HL})&=60ks_{A_5}(\op{HL})-20\cdot 3^g-15\cdot 4^g-24\cdot 5^g.
\end{align*} 
\end{corollary}

For the sake of convenience, we let $\mathbf{ks}_{G}(G')$ denote 
the set of conjugacy classes of homomorphisms from $G'$ to 
$G$; especially,
we have $ks_{G}(\op{HL})=\vert \mathbf{ks}_G(G_{\op{HL}})\vert$. 
 
\begin{lemma}\label{lm:A4_rep_two_generator_knot}
For a $2$-generator knot $K$, $ks_{A_4}(K)=4$ or $6$.
In each case, $\mathbf{ks}_{A_4}(G_K)$ contains four conjugacy classes
represented by homomorphisms whose images are abelian. 
If $ks_{A_4}(K)=6$, the two additional conjugacy classes
are represented by surjective homomorphisms.
\end{lemma}
\begin{proof}
Since any non-surjective homomorphism $\phi:G_K\rightarrow A_4$
factors throught the abelianization of $G_K$, 
$\op{Im}(\phi)$ is either trivial or isomorphic to 
$\mathbb{Z}_2$ or $\mathbb{Z}_3$. 
By \eqref{eq:formulas_ksA4}, the number of conjugacy classes
of non-surjective homomorphisms are
\[ks_{V_4}^{A_4}(K)+ks_{\mathbb{Z}_3}^{A_4}(K)-1=3+2-1=4,\]
and hence $ks_{A_4}(K)\geq 4$.

Now, consider a two-generator presentation of $G_K$  
\begin{equation}\label{eq:presentation_G_K}
<a,b\mid w(a,b)=1>
\end{equation}
and its abelianization: 
\begin{equation}\label{eq:abelianization}
G_K\xrightarrow{\pi} G_K/[G_K,G_K]\simeq \mathbb{Z}=<g>;
\end{equation}
let $g^{3n+l}, g^{3n'+l'}$ be the image of 
$a,b$ under \eqref{eq:abelianization}, respectively. 
Suppose both $l$ and $l'$ are non-zero,
then either $3\mid l'-l$ or $3\mid l'-2l$.
If $3\mid l'-l$, we replace $b$ with $b'$ by $b'=a^{-1}b$;
this implies a new presentation of $G_K$: 
\[G_K=<a,b'\mid w'(a,b')=1>,\]
where $w'(a,b')=w(a,ab')$, and $b'$ vanishes under the composition 
\[
G_K\xrightarrow{\pi} G_K/[G_K,G_K]\simeq \mathbb{Z} \xrightarrow{\pm} \mathbb{Z}_3\simeq A_4/[A_4,A_4]. 
\]
Similarly, if $3\mid 2l-l'$, we replace $b$ with $b''$ by $b''=a^{-2}b$ 
to get a new presentation 
\[G_K=<a,b''\mid w''(a,b'')=1>,\]
where $w''(a,b'')=w(a,a^2b'')$, and $b''$ vanishes under the composition 
\[
G_K\xrightarrow{\pi} G_K/[G_K,G_K]\simeq \mathbb{Z}\xrightarrow{\pm} \mathbb{Z}_3\simeq A_4/[A_4,A_4]. 
\]
 
Therefore, given a surjective homomorphism $\phi$,  
we may assume $\phi(b)$ in \eqref{eq:presentation_G_K} 
is in the commutator of $A_4$ and of order $2$
and $\phi(a)$ is of order $3$.
Up to conjugation, there are only two such homomorphisms: one
corresponds to $\phi(a)=(123)$, the other $\phi(a)=(132)$; note that
every two elements of order $2$ in $A_4$ are conjugate 
with respect to $(123)$ or $(132)$.  
This shows there are at most 
two surjective homomorphisms
from $G_K$ to $A_4$, and they always appear in pairs
because there exists an automorphism of $A_4$
sending $(123)$ to $(132)$, namely
\begin{align}
\Phi_{(23)}:A_4&\rightarrow A_4 \nonumber \\
     x&\mapsto (23)x(23),\label{eq:appear_in_pairs}
\end{align}
  
\end{proof}

\begin{lemma}\label{lm:A4_rep_two_generator_link}
If $L$ is a $2$-generator link, then $ks_{A_4}(L)$
is $14$, $16$, $18$, $20$ or $22$. In each case, 
$\mathbf{ks}_{A_4}(G_L)$ contains $14$ elements
represented by homomorphisms whose images are abelian.
If $ks_{A_4}(L)>14$, then
any additional conjugacy class is represented by
surjective homomorphisms. 
\end{lemma}
\begin{proof}
Suppose $\phi:G_L\rightarrow A_4$ is non-surjective, then it 
factors through the abelianization of $G_L$,
so by \eqref{eq:formulas_ksA4}, the number of 
conjugacy classes of non-surjective homomorphism can be computed by
\[ks_{V_4}^{A_4}(K)+ks_{\mathbb{Z}_3}^{A_4}(K)-1=9+6-1=14,\]
and particularly, $ks_{A_4}(L)\geq 14$.

Suppose $\phi:G_L\rightarrow A_4$ is onto, and 
\[<a,b\mid w(a,b)=1>\]  
is a presentation of $G_L$.
Then either both $\phi(a)$ and $\phi(b)$
are of order $3$ or one of them is of order $3$ and the other order $2$.
In the former case, up to conjugation, there are four possibilities:
\begin{align*}
\op{I}:\phi(a)&=(123),&\quad \phi(b)=(124);\\
\op{II}:\phi(a)&=(123),&\quad \phi(b)=(142);\\
\op{III}:\phi(a)&=(132),&\quad \phi(b)=(124);\\
\op{IV}: \phi(a)&=(132),&\quad \phi(b)=(142).
\end{align*}
By \eqref{eq:appear_in_pairs} 
$w(\phi(a),124)=1$ if and only if
$w(\Phi_{(23)}\big(\phi(a)\big),(142))=1$
since 
\[w(\Phi_{(23)}(\phi(a)),(124))=\Phi_{(23)}\big( w(\phi(a),(134))\big)\\
=\Phi_{(23)}\Big( (123)w\big(\phi(a),(142)\big)(132) \Big).
\] 
Therefore, I and IV appear in pair; so do II and IV, for a similar reason.
Now, if one of $\phi(a)$ and $\phi(b)$ is of order $2$, 
we also have four possibilities:
\begin{align*}
\op{I}':\phi(a)&=(123), &\quad \phi(b)&=(12)(34);\\
\op{II}':\phi(a)&=(132), &\quad \phi(b)&=(12)(34);\\
\op{III}':\phi(a)&=(12)(34), &\quad \phi(b)&=(123);\\
\op{IV}': \phi(a)&=(12)(34), &\quad \phi(b)&=(132).
\end{align*}
They appear in pairs as in the previous case. 
Thus, $ks_{A_4}(L)$ is an even integer between $14$ and $22$.
\end{proof} 

\subsection{Necessary conditions for reducibility}
We divide the proof of 
Theorems \ref{intro:thm:A4_criteria} and \ref{intro:thm:A5_criteria}
into three lemmas.

 
\begin{lemma}\label{lm:criterion_trivial_knot}
Given a reducible handlebody link $\op{HL}$ of genus $g$,
if the trivial knot is a factor of some factorization of 
$\op{HL}$,
then 
\[12 \mid ks_{A_4}(\op{HL})+6\cdot 3^{g-1}+2\cdot 4^{g-1}\quad \textbf{and}\quad
60 \mid ks_{A_5}(\op{HL})+14\cdot 4^{g-1}+19\cdot 3^{g-1} +22\cdot 5^{g-1}.
\]  
\end{lemma}
\begin{proof}
By the assumption, the knot group $G_{\op{HL}}$
is isomorphic to the free product $\mathbb{Z}\ast G_{\op{HL}'}$,
where $\op{HL}'$ is a handlebody link of genus $g-1$. 

Recall that $\mathbf{ks}_{A_4}(\mathbb{Z})$ contains four elements by \eqref{eq:formulas_ksA4}; let $\phi_1,\phi_{2},\phi_{3}^1,\phi_{3}^2$
be homomorphism representing these four conjugacy classes  
with $\op{Im}(\phi_1)$ trivial, $\op{Im}(\phi_{2})$ isomorphic to $\mathbb{Z}_2$, and 
$\op{Im}(\phi_{3}^i),i=1,2$ isomorphic to $\mathbb{Z}_3$. 
Then observe that, given a homomorphism $\phi:G_{\op{HL}}\rightarrow A_4$;
by conjugating with some elements in $A_4$,
we may assume its restriction $\phi\vert_{\mathbb{Z}}$
is one of 
\[\{\phi_1,\phi_{2},\phi_{3}^1,\phi_{3}^2\}.\]

\noindent
\textbf{Case 1: $\phi\vert_{\mathbb{Z}}=\phi_1$.} 
Let $\phi,\psi:G_{\op{HL}}\rightarrow A_4$
be two homomorphisms with 
\[\phi\vert_{\mathbb{Z}}=\psi\vert_{\mathbb{Z}}=\phi_1.\]
Then they are in the same conjugacy class if and only if 
their restrictions $\phi\vert_{G_{\op{HL}'}},\psi\vert_{G_{\op{HL}'}}$
are conjugate, so there are $ks_{A_4}(\op{HL}')$ conjugacy classes in Case 1. 

\noindent
\textbf{Case 2: $\phi\vert_{\mathbb{Z}_2}=\phi_{2}$.}
Let $\phi,\psi:G_{\op{HL}}\rightarrow A_4$
be two homomorphisms with 
\[\phi\vert_{\mathbb{Z}}=\psi\vert_{\mathbb{Z}}=\phi_{2}.\]
Then they are in the same conjugacy class if and only if 
\[\phi\vert_{G_{\op{HL}'}}=g\cdot \psi\vert_{G_{\op{HL}'}}\cdot g^{-1},\textbf{ for some }g\in V_4.\]
Hence in case 2, the number of conjugacy classes is
\[\frac{ks^w_{A_4}(\op{HL}')-ks^w_{V_4}(\op{HL}')}{4}+ks^w_{V_4}(\op{HL}').\]

\noindent
\textbf{Case 3: $\phi\vert_{\mathbb{Z}}=\phi_{3}^i$, $i=1$ or $2$.}
Let $\phi,\psi:G_{\op{HL}}\rightarrow A_4$
be two homomorphisms with 
\[\phi\vert_{\mathbb{Z}}=\psi\vert_{\mathbb{Z}}=\phi_{3}^i, i=1 
\textbf{(resp. $2$)}.\]
Then they are in the same conjugacy class if and only if 
\[\phi\vert_{G_{\op{HL}'}}=g\cdot \psi\vert_{G_{\op{HL}'}}\cdot g^{-1},
\textbf{ for some } 
g\in \op{Im}(\phi_{3}^i),i=1 \textbf{(resp. $2$)},\] 
and therefore for each $i$, there are 
\[\frac{ks^w_{A_4}(\op{HL}')-ks^w_{\mathbb{Z}_3}(\op{HL}')}{3}+ks^w_{\mathbb{Z}_3}(\op{HL}')\]
conjugacy classes.

Summing the three cases up gives
the formula of $ks_{A_4}(\op{HL})$ in terms of the $ks$-invariants
of $\op{HL}'$:
\begin{multline}\label{eq:formula_1sum_w_trivial_knot_A4}
ks_{A_4}(\op{HL})=ks_{A_4}(\op{HL}')
+\frac{ks^w_{A_4}(\op{HL}')-ks^w_{V_4}(\op{HL}')}{4}+ks^w_{V_4}(\op{HL}')\\
+2\cdot\left(\frac{ks^w_{A_4}(\op{HL}')-ks^w_{\mathbb{Z}_3}(\op{HL}')}{3}+ks^w_{\mathbb{Z}_3}(\op{HL}')\right).
\end{multline}
Combining \eqref{eq:formula_1sum_w_trivial_knot_A4} with 
\eqref{eq:formulas_ksw} and Corollary \ref{cor:formuala_ks_ksw},
we get the equation
\[
ks_{A_4}(\op{HL})=12\cdot ks_{A_4}(\op{HL}')-6\cdot 3^{g-1}-2\cdot 4^{g-1},
\]
which implies the first assertion.

$ks_{A_5}(\op{HL})$ can be computed in a similar manner.
First note that $\mathbf{ks}_{G}(\mathbb{Z})$ 
contains five elements by \eqref{eq:formulas_ksA5}, and 
they are represented by homomorphisms
\begin{equation}\label{eq:representing_homo_A5}
\phi_1,\phi_{2},\phi_3,\phi_5^1,\phi_5^2,
\end{equation}
with $\op{Im}(\phi_1)$ trivial, 
$\op{Im}(\phi_{2})$ isomorphic to $\mathbb{Z}_2$,
$\op{Im}(\phi_{3})$ isomorphic to $\mathbb{Z}_3$, and
$\op{Im}(\phi_{5}^{i}),i=1,2$, isomorphic to $\mathbb{Z}_5$.
As with the case of $A_4$,   
given a homomorphism 
$\phi:G_{\op{HL}}\rightarrow A_5$,
by conjugating with some element in $A_5$, 
we may assume its restriction on $\mathbb{Z}$ 
is one of the representing homomorphisms in \eqref{eq:representing_homo_A5}.
The number of conjugacy classes of 
homomorphisms that restrict to $\phi_1$ is $ks_{A_5}(L)$
and the number of conjugacy classes of 
homomorphisms that restrict to $\phi_{2},\phi_3$, or $\phi_5^i,i=1,2$, 
is 
\begin{align*}
&\frac{ks^w_{A_5}(\op{HL}')-ks^w_{V_4}(\op{HL}')}{4}+ks^w_{V_4}(\op{HL}'),\\ 
&\frac{ks^w_{A_5}(\op{HL}')-ks^w_{\mathbb{Z}_3}(\op{HL}')}{3}+ks^w_{\mathbb{Z}_3}(\op{HL}'), \\
\textbf{ or    }&
\frac{ks^w_{A_5}(\op{HL}')-ks^w_{\mathbb{Z}_5}(\op{HL}')}{5}+ks^w_{\mathbb{Z}_5}(\op{HL}'),
\end{align*}
respectively, and summing them up givues the formula of $ks_{A_5}(\op{HL})$: 
\begin{multline}\label{eq:formula_1sum_w_trivial_knot_A5}
ks_{A_5}(\op{HL})
=ks_{A_5}(\op{HL}')
+\frac{ks^w_{A_5}(\op{HL}')-ks^w_{V_4}(\op{HL}')}{4}+ks^w_{V_4}(\op{HL}')\\
+\frac{ks^w_{A_5}(\op{HL}')-ks^w_{\mathbb{Z}_3}(\op{HL}')}{3}+ks^w_{\mathbb{Z}_3}(\op{HL}')\\
+2\cdot 
\left(\frac{ks^w_{A_5}(\op{HL}')-ks^w_{\mathbb{Z}_5}(\op{HL}')}{5}
+ks^w_{\mathbb{Z}_5}(\op{HL}')\right). 
\end{multline}   
The formula \eqref{eq:formula_1sum_w_trivial_knot_A5}, together with
\eqref{eq:formulas_ksw} and Corollary \ref{cor:formuala_ks_ksw},
implies the identity:
\[ks_{A_5}(\op{HL})=60\cdot ks_{A_5}(\op{HL}')-19\cdot 3^{g-1}-14\cdot 4^{g-1}- 22\cdot 5^{g-1},\]
and thus the second assertion.
\end{proof}

\begin{lemma}\label{lm:criterion_2-generator_knot}
Given a reducible handlebody link $\op{HL}$ of genus $g$,
if a $2$-generator knot $K$ is a factor of some factorization of $\op{HL}$,
then 
\[12+24k \mid ks_{A_4}(\op{HL})+(6+16k)\cdot 3^{g-1}+ (2+6k)\cdot 4^{g-1},\]
where $k=0$ or $1$.
\end{lemma}
\begin{proof}
By the assumption the knot group $G_{\op{HL}}$  
is isomorphic to the free product  
$G_K\ast G_{\op{HL}'}$,
where $\op{HL}'$ is a handlebody link of genus $g-1$.
By Lemma \ref{lm:A4_rep_two_generator_knot},  
$\mathbf{ks}_{A_4}(G_K)$ might have two more elements
than $\mathbf{ks}_{A_4}(\mathbb{Z})$. Let 
$\phi_s^1,\phi_s^2$ be representing surjective homomorphisms
of these two conjugacy classes. Then,
since two homomorphisms  
\begin{equation}\label{eq:restrict_to_phi_s}
\phi,\psi:G_{\op{HL}}\rightarrow A_4 \quad\textbf{ with }\quad\phi\vert_{G_K}=\psi\vert_{G_K}=\phi_s^i, \quad\textbf{$i=1$ or $2$} 
\end{equation}
are conjugate if and only if 
\[\phi\vert_{G_{\op{HL}}'}=\psi\vert_{G_{\op{HL}'}}.\]
there are $ks_{A_4}^{w}(\op{HL}')$ conjugacy classes of 
homomorphisms with the property \eqref{eq:restrict_to_phi_s}. 
Adding this to \eqref{eq:formula_1sum_w_trivial_knot_A4},
we obtain
\begin{multline}\label{eq:formula_1sum_w_two_ge_knot_A4}
ks_{A_4}(\op{HL})=ks_{A_4}(\op{HL}')
+\frac{ks^w_{A_4}(L)-ks^w_{V_4}(\op{HL}')}{4}+ks^w_{V_4}(\op{HL}')\\
+2\cdot\left(\frac{ks^w_{A_4}(\op{HL}')-ks^w_{\mathbb{Z}_3}(\op{HL}')}{3}+ks^w_{\mathbb{Z}_3}(\op{HL}')\right)
+2k\cdot ks_{A_4}^w(\op{HL}'),
\end{multline}
where $k=0$ or $1$. Plugging \eqref{eq:formulas_ksw} and Corollary \ref{cor:formuala_ks_ksw} into 
\eqref{eq:formula_1sum_w_two_ge_knot_A4} 
implies the identity: 
\[ks_{A_4}(\op{HL})
=(12+24k)\cdot ks_{A_4}(\op{HL}')-(6+16k)\cdot 3^{g-1}-(2+6k)\cdot 4^{g-1},
\textbf{$k=0$ or $1$},\]
and therefore the assertion. 
\end{proof}

\begin{lemma}
Given a reducible handlebody link $\op{HL}$ of genus $g$,
if a $2$-generator link $L$ is a factor of some factorization of $\op{HL}$,
then 
\[48+24k \mid ks_{A_4}(\op{HL})+(26+16k)\cdot 3^{g-2} + (8+6k)\cdot 4^{g-2},\]
where $k=0,1,2,3,$ or $4$.
\end{lemma}
\begin{proof}
By the assumption, the knot group
$G_{\op{HL}}$ is isomorphic to the free product $G_L\ast G_{\op{HL}'}$,
where $\op{HL}'$ is
a handlebody link of genus $g-2$. By Lemma \ref{lm:A4_rep_two_generator_link}, 
$\mathbf{ks}_{A_4}(G_L)$ contains $14+2k$ elements, $k=0,1,2,3$, or $4$,
where one conjugacy class for the trivial homomorphism,
five for non-trivial homomorphisms whose images are in $V_4$, eight for
homomorphisms whose images isomorphic to $\mathbb{Z}_3$, and
$2k$ for surjective homomorphisms. The same argument as in the proof of Lemmas \ref{lm:criterion_trivial_knot} and \ref{lm:criterion_2-generator_knot}
gives  
\begin{multline}\label{eq:formula_1sum_w_two_ge_link_A4}
ks_{A_4}(\op{HL})=ks_{A_4}(\op{HL}')
+5\cdot\left(\frac{ks^w_{A_4}(\op{HL}')-ks^w_{V_4}(\op{HL}')}{4}+ks^w_{V_4}(\op{HL}')\right)\\
+8\cdot\left(\frac{ks^w_{A_4}(\op{HL}')-ks^w_{\mathbb{Z}_3}(\op{HL}')}{3}+ks^w_{\mathbb{Z}_3}(\op{HL}')\right)
+2k\cdot ks_{A_4}^w(\op{HL}'),
\end{multline}
where $k=0, 1, 2, 3$, or $4$.  Plugging 
\eqref{eq:formulas_ksw} and Corollary \ref{cor:formuala_ks_ksw}
into
\eqref{eq:formula_1sum_w_two_ge_link_A4}, we obtain
\[ks_{A_4}(\op{HL})
=(48+24k)\cdot ks_{A_4}(\op{HL}')-(26+16k)\cdot 3^{g-2}-(8+6k)\cdot 4^{g-2}\]
and hence the lemma.
\end{proof}

\section{Examples}\label{sec:examples}
\subsection{Applications to handlebody knot/link tables}
Irreducibility of handlebody knots in \cite{IshKisMorSuz:12} 
and handlebody links in \cite{BePaPaWa:20} are examined here
with the irreducibility criteria 
(Corollary \ref{intro:cor:A4-A5_criteria_HK} and Table
\ref{intro:tab:applications_links}).
The $ks_{A_4}$-and $ks_{A_5}$-invariants of handlebody links
are computed by the Appcontour \cite{appcontour};
the same software is also used to find an upper bound of 
the rank of each knot group. In many cases, the 
upper bound is identical to the rank.   

\begin{table}[ht]  
\caption{Irreducibility of 
Ishii, Kishimoto, Moriuchi and Suzuki's handlebody knots}
\begin{tabular}{c|c|c|c|c|c}
handlebody knot & rank  & $ks_{A_4}$ & $A_4$-criterion \eqref{intro:eq:criterion_trivial_knot} & $ks_{A_5}$ & $A_5$-criterion
\eqref{intro:eq:criterion_trivial_knot_A5}\\ \hline
$\op{HK}4_1$ & 3 & 30 &\checkmark & 156 & \checkmark \\ \hline
$\op{HK}5_1$ & 3 & 22 & ? & 111 & \checkmark \\ \hline 
$\op{HK}5_2$ & 3 & 30 & \checkmark & 156 & \checkmark \\ \hline  
$\op{HK}5_3$ & 3 & 30  & \checkmark & 105  & \checkmark \\ \hline  
$\op{HK}5_4$ & 3 & 22  & ? & 365 & \checkmark \\ \hline  
$\op{HK}6_1$ & 3 & 30  & \checkmark  & 143 & \checkmark \\ \hline  
$\op{HK}6_2$ & 3 & 30  & \checkmark & 105 & \checkmark \\ \hline  
$\op{HK}6_3$ & 3 & 22  & ?  & 83 & \checkmark \\ \hline  
$\op{HK}6_4$ & 3 & 22  & ? & 111 & \checkmark \\ \hline   
$\op{HK}6_5$ & 3 & 22  & ? & 97 & \checkmark \\ \hline   
$\op{HK}6_6$ & 3 & 22  & ? & 97 & \checkmark \\ \hline   
$\op{HK}6_7$ & 3 & 30  & \checkmark & 157 & \checkmark \\ \hline   
$\op{HK}6_8$ & 3 & 22    & ? & 105 & \checkmark \\ \hline   
$\op{HK}6_{9}$ & 3 & 30  & \checkmark & 146 & \checkmark \\ \hline   
$\op{HK}6_{10}$ & 3 & 22 & ? & 195 & \checkmark \\ \hline   
$\op{HK}6_{11}$ & 3 & 22 & ? & 73 & \checkmark \\ \hline   
$\op{HK}6_{12}$ & 3 & 30 & \checkmark & 135 & \checkmark \\ \hline   
$\op{HK}6_{13}$ & 3 & 30 & \checkmark & 156 & \checkmark \\ \hline   
$\op{HK}6_{14}$ & 3 & 46 & ?& 353 & \checkmark \\ \hline   
$\op{HK}6_{15}$ & 3 & 46 & ?& 353 & \checkmark \\ \hline   
$\op{HK}6_{16}$ & 3 & 22 & ? & 267 & \checkmark \\ \hline   
\end{tabular}
\label{tab:reducibility_HK}
\end{table}

The results of the irreducibility test
are recorded in Tables \ref{tab:reducibility_HK} and \ref{tab:reducibility_HL}, 
where the check mark $\checkmark$ stands for the corresponding 
condition(s) not satisfied, 
and hence the handlebody link is irreducibile, and the question
mark means the opposite, so its irreducibility is inconclusive.
To avoid confusion, $\op{HK}$ is added to the name of
each handlebody knot in \cite{IshKisMorSuz:12}; 
so is $\op{HL}$ to the name of each handlebody link in \cite{BePaPaWa:20}. 
 
Since all handlebody knots in \cite{IshKisMorSuz:12} are 
$3$-generator, by Corollary \ref{intro:cor:A4-A5_criteria_HK}, 
if either $12$ does not divide
$ks_{A_4}(\op{HK})+26$, or $60$ does not divide 
$ks_{A_5}(\op{HK})+223$, $\op{HK}$ is irreducible.
On the contrary, in Table \ref{tab:reducibility_HL} 
different criteria are required to test each case, depending on the rank 
and the number of component (the column ``comp.")
based on Table \ref{intro:tab:applications_links}.
For instance, for a $3$-generator handlbody link of type $[1,1]$, such as 
$\op{HL}4_1$, if it fails either of \eqref{intro:eq:criterion_trivial_knot}
and \eqref{intro:eq:criterion_trivial_knot_A5}, it is irreducible.
But, for $\op{HL}5_1$, which is possibly $4$-generator, 
we need to have \emph{both} \eqref{intro:eq:criterion_trivial_knot} 
and \eqref{intro:eq:criterion_2-gen_knot} failed in order to draw a conclusion; 
also, the $A_5$ criterion is not applicable in this case.

\begin{table}[ht]  
\caption{Irreducibility of handlebody links in \cite{BePaPaWa:20}}
\begin{tabular}{c|c|c|c|c|c|c}
comp. & handlebody link & rank  & $ks_{A_4}$ & $A_4$-criterion   & $ks_{A_5}$ & $A_5$-criterion\\ \hline
\multirow{10}{*}{2}&$\op{HL}4_1$ & 3 & 114 
&\checkmark
& 600 & \checkmark \\ \cline{2-7}
&$\op{HL}5_1$ & $\leq 4$ & 98 & 
\checkmark & 660 & not applicable\\ \cline{2-7} 
&$\op{HL}6_1$ & 3 & 90  & \checkmark  & 600 & \checkmark \\ \cline{2-7}  
&$\op{HL}6_2$ & 3 & 106  & ?
 & 689 & \checkmark \\ \cline{2-7}  
&$\op{HL}6_3$ & 3 & 90  & \checkmark  & 469 & \checkmark \\ \cline{2-7}  
&$\op{HL}6_4$ & 3 & 106  & 
? & 689 & \checkmark \\ \cline{2-7}
&$\op{HL}6_5$ & $\leq 4$ & 210  & \checkmark & 4020 & not applicable \\ \cline{2-7}   

&$\op{HL}6_6$ & 3 & 130  & ? & 1380 & \checkmark \\ \cline{2-7}   
&$\op{HL}6_7$ & $\leq 4$ & 98  & \checkmark & 597 & not applicable \\ \cline{2-7}   
&$\op{HL}6_8$ & 3 & 114    & \checkmark & 1401 & \checkmark \\ \hline
\multirow{6}{*}{3}
&$\op{HL}6_{9}$ & 4 & 310  & ? & 1841 & not applicable \\ \cline{2-7}   
&$\op{HL}6_{10}$ & 4 & 326 & \checkmark & 2636 & not applicable \\ \cline{2-7}   
&$\op{HL}6_{11}$ & 4 & 486 & \checkmark & 5876 & not applicable \\ \cline{2-7} 
&$\op{HL}6_{12}$ & 4 & 502 & ? & 5883 & not applicable \\ \cline{2-7} 
&$\op{HL}6_{13}$ & 4 & 822 & \checkmark & 19308 & not applicable \\ \cline{2-7}   
&$\op{HL}6_{14}$ & 4 & 486 & \checkmark& 5876 & not applicable \\ \hline   
4 &$\op{HL}6_{15}$ & 5 &  1242 & \checkmark & 12072  & not applicable\\ \hline   
\end{tabular}
\label{tab:reducibility_HL}
\end{table}

\subsection{Irreducible handlebody links of a given type}
Here we present a construction of irreducible handlebody link
of any given type. First we introduce the notion of 
$\mathcal{D}$-irreducibility for handlebody-link-disk pairs.
\begin{definition}[\textbf{$\mathcal{D}$-irreducibility}]
A handlebody link $\op{HL}$ is $\mathcal{D}$-irreducible if either
its complement $\Compl{\op{HL}}$ 
admits no incompressible disks or it is a trivial knot.
A handlebody-link-disk pair $(\op{HL},D)$ is a handlebody link 
$\op{HL}$ together with an incompressible disk $D\subset \op{HL}$.
The pair $(\op{HL},D)$ is $\mathcal{D}$-irreducible 
if there exists no incompressible disk $D'$ in the complement $\Compl{\op{HL}}$
with $D'\cap D=\emptyset$. 
An unknot with a meridian disk is 
the trivial $\mathcal{D}$-irreducible handlebody-link-disk pair. 
\end{definition}
$\mathcal{D}$-irreducibility is equivalent to 
irreducibility for genus $g\leq 2$ handlebody knots \cite{Tsu:75} but
stronger in general \cite[Examples $5.5$-$6$]{Suz:75}, \cite[Remark $3.3$]{BePaPaWa:20}. Any $\mathcal{D}$-irreducible handlebody link
with an incompressible disk is a $\mathcal{D}$-irreducible pair. 
On the other hand, the underlying handlebody link of 
a $\mathcal{D}$-irreducible handlebody-link-disk pair could be trivial 
(left handlebody-knot-disk pair in Fig.\ \ref{fig:knotsum_w_L}).     
%
%
%
%
\begin{definition}[\textbf{Knot sum}]\label{def:twosum}
The knot sum  
of two handlebody-link-disk pairs  
$(\op{HL}_1,D_1), (\op{HL}_2,D_2)$ is a handlebody link 
$(\op{HL}_1,D_1)\# (\op{HL}_2,D_2)$ obtained
by gluing $\op{HL}_1,\op{HL}_2$ together as follows:
first remove a $3$-ball $B_1$ (resp.\ $B_2$) 
with $\mathring{B}_1\cap \op{HL}_1$ $(\text{resp. } \mathring{B}_2\cap \op{HL}_2)$ 
a tubular neighborhood $N(D_1)$ of $D_1$ $(\text{resp. }N(D_2)$ of $D_2)$ 
from $\sphere$, where $\overline{N(D_1)}$ $(\text{resp. }\overline{N(D_2)})$ 
can be identified with the oriented $3$-manifold 
$D_1\times [0,1]$ $(\text{resp. }D_2\times [0,1])$ 
using the given orientation on $D_1$ $(\text{resp. }D_2)$.
Then the knot sum is given by gluing resultant $3$-manifolds
$\Compl{B_1}, \Compl{B_2}$ via an orientation-reversing homeomorphism
$f:\partial B_1\rightarrow \partial B_2$ 
with $f(D_1\times \{i\})=D_2\times \{j\}$, $i-j\equiv 1$ mod $2$.  
\end{definition}
\begin{figure}[ht]
\def\svgwidth{0.9\columnwidth}
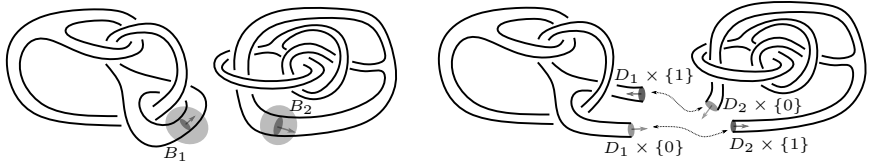   
\caption{Knot sum of $\op{HK}4_1$ and $\op{HK}5_1$ with meridian disks} 
\label{fig:knotsum}
\end{figure}
The knot sum resembles the order-$2$ connected sum of spatial graphs \cite{Mor:07}.
%
%
\begin{theorem}\label{thm:knotsum_irreduciblepairs}
The knot sum 
of two non-trivial $\mathcal{D}$-irreducible handlebody-link-disk pairs 
$(\op{HL}_1,D_1),(\op{HL}_2,D_2)$
is $\mathcal{D}$-irreducible. 
\end{theorem}
\begin{proof}
We prove by contradiction.
Suppose the knot sum
\[\op{HL}\simeq (\op{HL}_1,D_1)\# (\op{HL}_2,D_2)\]
is not $\mathcal{D}$-irreducible, and $D$ is an incompressible disk in
$\Compl{\op{HL}}$. 

Let $B$ be the $3$-ball such that
$B\cap\Compl{\op{HL}}$ is the complement of $\op{HL}_2$, and 
denote the intersection annulus $\Compl{\op{HL}}\cap\partial B$
by $A$. Isotopy $D$ such that the number of components of 
$A\cap D$ is minimized.  
  
\noindent
\textbf{Claim: $A\cap D=\emptyset$.}
Suppose the intersection is non-empty, 
then we can choose a component $\alpha$ of $A\cap D$
that is innermost in $D$. $\alpha$ must be an arc, for otherwise
it would contradict either the $\mathcal{D}$-irreducibility 
of $(\op{HL}_i,D_i)$ or the minimality.
$\alpha$ cuts $D$ into two disks, one of which,
say $D'$, has no intersection with $A$. 
Without loss of generality, we may assume $D'$
is in $\Compl{B}$.

If $\alpha$ is essential in $A$, then $\op{HL}_1$
is equivalent to 
the union of a tubular neighborhood of $\alpha$ in $B$
and $\Compl{B}\cap \op{HL}$ in $\sphere$.
Since $D'\cap \partial D$ is an arc connecting 
two sides of $D_1$ in $\op{HL}_1$, $D_1$
is not separating and therefore a meridian disk of $\op{HL}_1$. 
In addition, $D'$ and $\partial D_1$ intersect at only one point,
so $(\op{HL}_1,D_1)$ is either trivial or not $\mathcal{D}$-irreducible,
contradicting the assumption.

If $\alpha$ is inessential in $A$, let $D''$ be the disk cut off from $A$
by $\alpha$. Then $D'\cup D''$ is a compressing disk in $\op{HL}_1$. 
If $\partial (D'\cup D'')$ is inessential in $\partial \op{HL}_1$,
the intersection $\alpha$ can be
removed---with other intersection arcs intact---by isotopying $A$.
On the other hand, the $\mathcal{D}$-irreducibility of 
$(\op{HL}_1,D_1)$ forces $\partial (D'\cup D'')$ 
to be inessential in $\partial \op{HL}_1$.
Thus, we have proved the claim, 
from which the theorem follows readily.
\end{proof}
\begin{figure}[ht]
\centering
\begin{subfigure}{.4\linewidth}
\captionsetup{font=footnotesize,labelfont=footnotesize}
\def\svgwidth{0.9\columnwidth}
\begingroup%
  \makeatletter%
  \providecommand\color[2][]{%
    \errmessage{(Inkscape) Color is used for the text in Inkscape, but the package 'color.sty' is not loaded}%
    \renewcommand\color[2][]{}%
  }%
  \providecommand\transparent[1]{%
    \errmessage{(Inkscape) Transparency is used (non-zero) for the text in Inkscape, but the package 'transparent.sty' is not loaded}%
    \renewcommand\transparent[1]{}%
  }%
  \providecommand\rotatebox[2]{#2}%
  \newcommand*\fsize{\dimexpr\f@size pt\relax}%
  \newcommand*\lineheight[1]{\fontsize{\fsize}{#1\fsize}\selectfont}%
  \ifx\svgwidth\undefined%
    \setlength{\unitlength}{1360.62992126bp}%
    \ifx\svgscale\undefined%
      \relax%
    \else%
      \setlength{\unitlength}{\unitlength * \real{\svgscale}}%
    \fi%
  \else%
    \setlength{\unitlength}{\svgwidth}%
  \fi%
  \global\let\svgwidth\undefined%
  \global\let\svgscale\undefined%
  \makeatother%
  \begin{picture}(1,0.5)%
    \lineheight{1}%
    \setlength\tabcolsep{0pt}%
    \put(0.43663421,0.25125847){\color[rgb]{0,0,0}\makebox(0,0)[lt]{\lineheight{1.25}\smash{\begin{tabular}[t]{l}$L$\end{tabular}}}}%
    \put(0.35835076,0.25455224){\color[rgb]{0,0,0}\makebox(0,0)[lt]{\lineheight{1.25}\smash{\begin{tabular}[t]{l}{\footnotesize $\#$}\end{tabular}}}}%
    \put(0.12358549,0.00804379){\color[rgb]{0,0,0}\makebox(0,0)[lt]{\lineheight{1.25}\smash{\begin{tabular}[t]{l}{\tiny \textcolor{red}{incompressible disk}}\end{tabular}}}}%
    \put(0,0){\includegraphics[width=\unitlength,page=1]{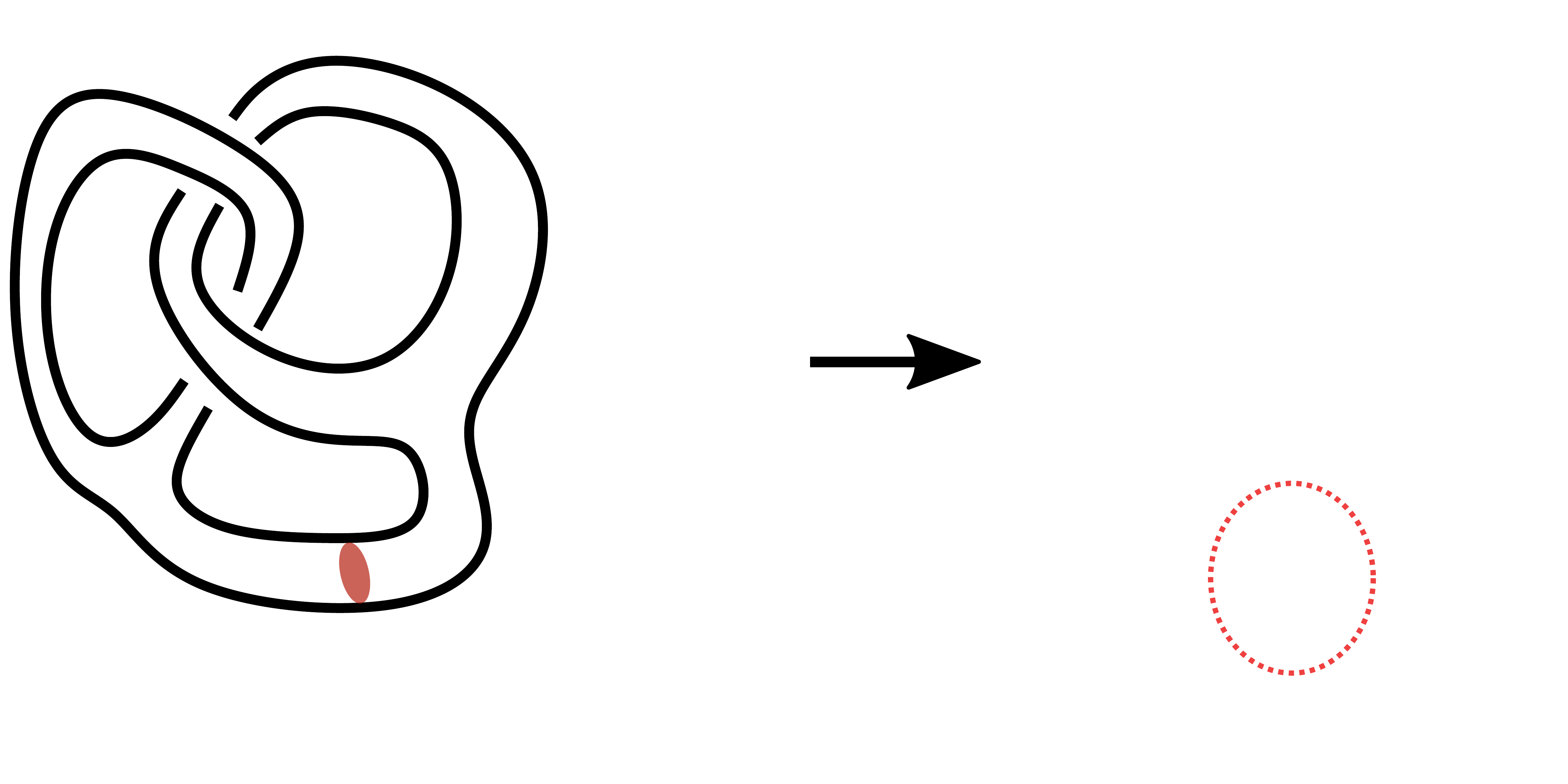}}%
    \put(0.80122887,0.11179529){\color[rgb]{0,0,0}\makebox(0,0)[lt]{\lineheight{1.25}\smash{\begin{tabular}[t]{l}$L$\end{tabular}}}}%
    \put(0,0){\includegraphics[width=\unitlength,page=2]{knot_sum_irred_pairs.pdf}}%
  \end{picture}%
\endgroup%
   
\caption{Knot sum with a link $L$}
\label{fig:knotsum_w_L}
\end{subfigure}
\begin{subfigure}{.5\linewidth}
\captionsetup{font=footnotesize,labelfont=footnotesize}
\def\svgwidth{0.95\columnwidth}
\begingroup%
  \makeatletter%
  \providecommand\color[2][]{%
    \errmessage{(Inkscape) Color is used for the text in Inkscape, but the package 'color.sty' is not loaded}%
    \renewcommand\color[2][]{}%
  }%
  \providecommand\transparent[1]{%
    \errmessage{(Inkscape) Transparency is used (non-zero) for the text in Inkscape, but the package 'transparent.sty' is not loaded}%
    \renewcommand\transparent[1]{}%
  }%
  \providecommand\rotatebox[2]{#2}%
  \newcommand*\fsize{\dimexpr\f@size pt\relax}%
  \newcommand*\lineheight[1]{\fontsize{\fsize}{#1\fsize}\selectfont}%
  \ifx\svgwidth\undefined%
    \setlength{\unitlength}{1417.32283465bp}%
    \ifx\svgscale\undefined%
      \relax%
    \else%
      \setlength{\unitlength}{\unitlength * \real{\svgscale}}%
    \fi%
  \else%
    \setlength{\unitlength}{\svgwidth}%
  \fi%
  \global\let\svgwidth\undefined%
  \global\let\svgscale\undefined%
  \makeatother%
  \begin{picture}(1,0.48)%
    \lineheight{1}%
    \setlength\tabcolsep{0pt}%
    \put(0,0){\includegraphics[width=\unitlength,page=1]{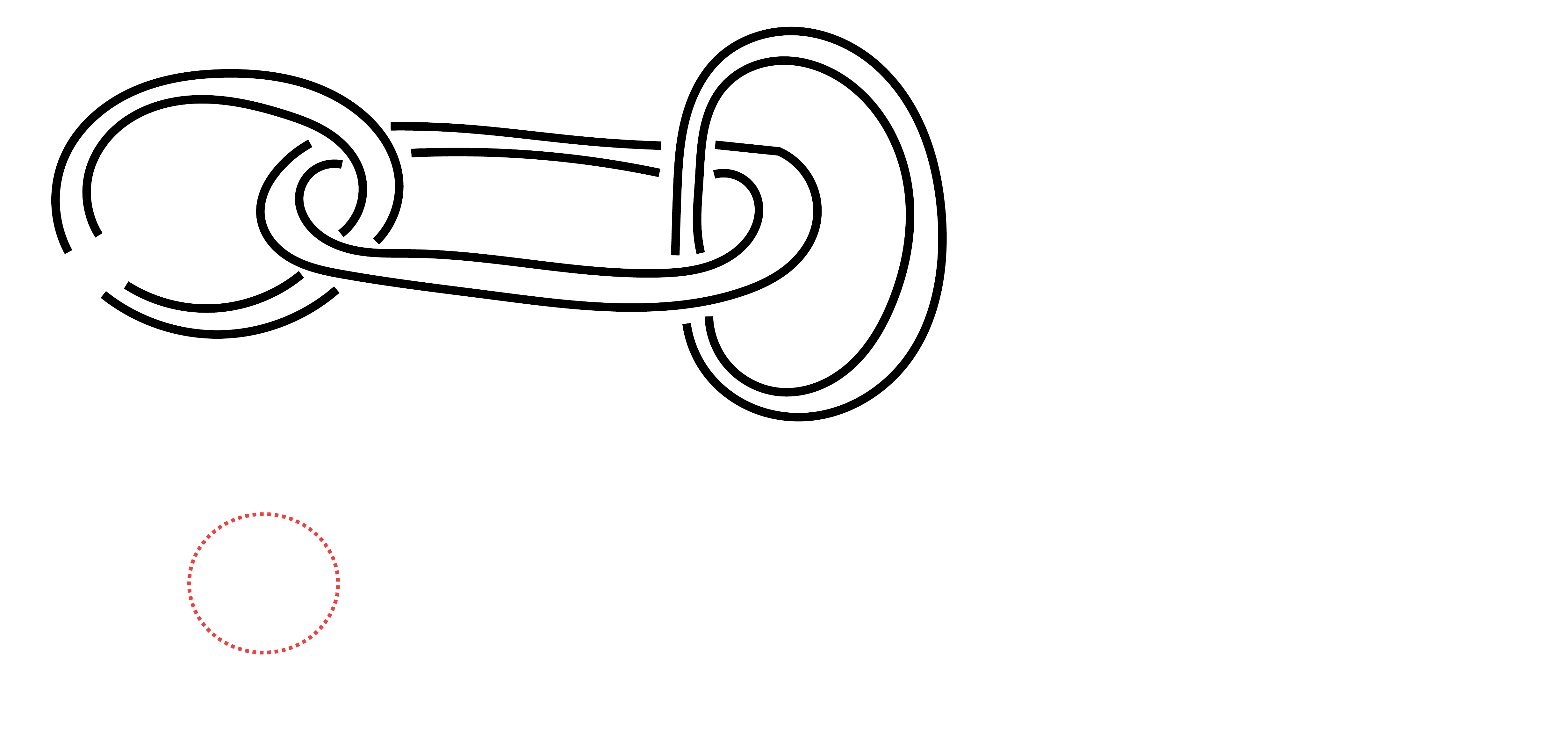}}%
    \put(0.13322867,0.09228731){\color[rgb]{0,0,0}\makebox(0,0)[lt]{\lineheight{1.25}\smash{\begin{tabular}[t]{l}$K_1$\end{tabular}}}}%
    \put(0,0){\includegraphics[width=\unitlength,page=2]{irred_hl_of_type_211.pdf}}%
    \put(0.5218765,0.04000818){\color[rgb]{0,0,0}\rotatebox{-1.44410375}{\makebox(0,0)[lt]{\lineheight{1.25}\smash{\begin{tabular}[t]{l}$K_2$\end{tabular}}}}}%
    \put(0,0){\includegraphics[width=\unitlength,page=3]{irred_hl_of_type_211.pdf}}%
    \put(0.91231677,0.34650262){\color[rgb]{0,0,0}\makebox(0,0)[lt]{\lineheight{1.25}\smash{\begin{tabular}[t]{l}$K_3$\end{tabular}}}}%
    \put(0,0){\includegraphics[width=\unitlength,page=4]{irred_hl_of_type_211.pdf}}%
    \put(0.18455104,0.44287108){\color[rgb]{0,0,0}\makebox(0,0)[lt]{\lineheight{1.25}\smash{\begin{tabular}[t]{l}{\tiny chain of rings}\end{tabular}}}}%
    \put(0,0){\includegraphics[width=\unitlength,page=5]{irred_hl_of_type_211.pdf}}%
  \end{picture}%
\endgroup%
   
\caption{Irreducibile handlebody link of type $[2,1,1]$}
\label{fig:irred_hl_of_a_given_type}
\end{subfigure}
\caption{}
\label{fig:examples_irred}
\end{figure}
In Fig.\ \ref{fig:examples_irred}, $K_1,K_2,K_3,L$
are knots or links; if $L$ in Fig.\ \ref{fig:knotsum_w_L} is 
the composition of two Hopf links,
the resulting knot sum is $\op{HL}6_{12}$. Hence its irreducibility,
which cannot be seen by our irreducibility test, follows from Theorem \ref{thm:knotsum_irreduciblepairs}. 
The following corollary 
generalizes Suzuki's example \cite[Theorem $5.2$]{Suz:75}.
\begin{corollary}
Given $m$ non-negative integers $n_1,n_2,\dots,n_m$ with $n:=\sum n_i>0$, 
there is an irreducible handlebody link of type $[n_1,n_2,\dots,n_m]$.
\end{corollary}
\begin{proof}
Consider a chain of rings with $n$-component---a knot sum
of $n-1$ Hopf links (Fig.\ \ref{fig:irred_hl_of_a_given_type}).
Label each ring with a number in $\{1,2,\dots, n\}$,
and for the ring with label $k$,  
\[\sum_{i=1}^{l-1}n_i<k\leq \sum_{i=1}^{l}n_i,\]
we consider its knot sum with 
an irreducible handlebody knot of genus $l$,
which can be obtained by performing the knot sum operation
iteratively on handlebody knots  
in \cite{IshKisMorSuz:12} with meridian disks 
(Fig.\ \ref{fig:knotsum}). The resultant handlebody link is necessarily
irreducible by Theorem \ref{thm:knotsum_irreduciblepairs}
and of the prescribed type.  
\end{proof}

\section*{Acknowledgements}
The paper is benefited from the support of National Center for Theoretical Sciences.
%


\begin{thebibliography}{99}
%
\nada{ 

\bibitem{Ale:24} J. W. Alexander, 
\textit{An example of a simply connected surface bounding a region which is not simply connected},
Proc. Natl.  Acad. Sci. USA
{\bf 10} (1924), 8--10.

\bibitem{Hat:83} A. Hatcher,
\textit{A proof of the Smale Conjecture}
Ann. of Math. {\bf 117} (1983), 553--607.

\bibitem{BeBePaPa:15}
G. Bellettini, V. Beorchia, M. Paolini, F. Pasquarelli,
\textit{Shape Reconstruction from Apparent Contours.
Theory and Algorithms,}
Computational Imaging and Vision,
Springer-Verlag (2015) pp. iii-333.


\bibitem{BeFrGh:12} 
R. Benedetti, R. Frigerio, R. Ghiloni,
\textit{The topology of Helmholtz domains},
Expo. Math. {\bf 30} (2012), 319--375


\bibitem{BePaWa:19}
G. Bellettini, M. Paolini, Y.-S. Wang,
\textit{On closed oriented surfaces in the 3-sphere},
$<$\url{arXiv:1902.05030 [math.GT]}$>$. 
}
 
\bibitem{BePaPaWa:20}
G. Bellettini, G. Paolini, M. Paolini, Y.-S. Wang:
\textit{Complete classification of $(n,1)$-handlebody links up to six crossings}, to appear. 

\nada{
\bibitem{Cer:68} J. Cerf,
\textit{Sur les diff\'eomorphismes de la sph\`ere de dimension trois ($\Gamma_4=0$)},
Springer-Verlag, Berlin and New York, Lecture Notes in Math. {\bf 53} (1968).
 

\bibitem{DavShe:01} R.J. Daverman, R.B. Sher,
\textit{Handbook of Geometric Topology}
North-Holland, Amsterdam (2001).

\bibitem{Dehn:11}
M. Dehn,
\textit{\"Uber unendliche diskontinuierliche Gruppen},
Math. Ann. {\bf 71} (1911), 116--144.


\bibitem{Eps:66} D. B. A. Epstein,
\textit{Curves on 2-manifolds and isotopies},
Acta Math. {\bf 115} (1966), 83--107.
 




\bibitem{FarMar:11} B. Farb, D. Margalit,
\textit{A Primer on Mapping Class Groups}, Princeton University Press, (2011).

\bibitem{FriMilPow:17} S. Friedl, A. N. Miller, M. Powell,
\textit{Determinants of amphichiral knots}, 
arXiv:1706.07940 [math.GT]. 


\bibitem{FomMat:97}
A. T. Fomenko, S. V. Matveev,
\textit{Algorithmic and Computer Methods for Three-Manifolds},
Kluwer Academic Publisher, (1997).




\bibitem{Fox:48} R. H. Fox,
\textit{On the imbedding of polyhedra in 3-space}, Ann. of Math. {\bf 2}(49) (1948), 462--470.


\bibitem{Fox:50} R. H. Fox,
\textit{Recent development of knot theory at Princeton}, Proceedings of the International Congress of Mathematicians
{\bf 2} (1950), 453--458.

\bibitem{Fox:52} R. H. Fox,
\textit{On the complementary domains of a certain pair of inequivalent knots}, Indagationes Mathematicae (Proceedings) {\bf 55} (1952), 37--40.

\bibitem{GoGu:74}
M. Golubitsky, V. Guillemin,
Stable Mappings and Their Singularities,
Grad. Texts in Math., 14, New York, Heidelberg,
Berlin, 1974. 

\bibitem{GorLue:89}
C. Gordon, J. Luecke,
\textit{Knots are determined by their complements}
J. Amer. Math. Soc. 
{\bf 2} (1989), 371--415.

\bibitem{Gro:69} J. L. Gross,
\textit{A unique decomposition theorem for $3$-manifold with connected boundary},
Trans. Amer. Math. Soc. {\bf 142} (1969), 191--199. 

\bibitem{Gro:70} J. L. Gross,
\textit{A unique decomposition theorem for $3$-manifold with several boundary components},
Trans. Amer. Math. Soc. {\bf 147} (1970), 561--572. 
}

\bibitem{Gru:40} I. A. Grushko:
\textit{On the bases of a free product of groups},
Matematicheskii Sbornik, {\bf 8} (1940), 169--182.



 
  

\nada{

\bibitem{Hei:71} W. Heil,
\textit{On Kneser's conjecture for bounded $3$-manifolds},
Proc. Comb. Phil. Soc. {\bf 71} (1972), 71--30.



\bibitem{Hem:04} J. Hempel,
\textit{3-manifolds}, AMS Chelsea Publishing, Providence, RI, (2004). 
}  
\bibitem{Ish:08} A. Ishii:
\textit{Moves and invariants for knotted handlebodies},
Algebr. Geom. Topol.
{\bf 8}, No. 3 (2008), 1403--1418.


 
\bibitem{IshKis:11}
A. Ishii, K. Kishimoto:
\textit{The quandle coloring invariant of a reducible handlebody-knot},
Tsukuba J. Math. {\bf 35} (2011), 131--141.
  
\bibitem{IshKisMorSuz:12}
A. Ishii, K. Kishimoto, H. Moriuchi, M. Suzuki:  
\textit{A table of genus two handlebody-knots up to six crossings},
J. Knot Theory Ramifications {\bf 21}
(2012),  1250035. 

\nada{
\bibitem{Joh:95} K. Johannson,
\textit{Topology and Combinatorics of 3-Manifolds},
Lecture Notes in Mathematics, Springer-Verlag Berlin Heidelberg,   
{\bf 1599} (1995).
} 
 
\bibitem{KitSuz:12} T. Kitano, M. Suzuki:
\textit{On the number of $SL(2,\mathbb{Z}/p\mathbb{Z})$-representations of knot groups}
J. Knot Theory Ramifications {\bf 21}
(2012), 1250035. 


\nada{
\bibitem{KorMiz:17} Y. Kotorii, A. Mizusawab,
\textit{HL-homotopy of handlebody-links and Milnor's invariants},
Topology Appl. {\bf 221} (2017), 715--736.


\bibitem{Kur:60} A. G. Kurosh,
\textit{The theory of groups}  
Translated from the Russian and edited by K. A. Hirsch. 2nd English ed. 2 volumes Chelsea Publishing Co., New York (1960) Vol. 1: 272 pp. Vol. 2: 308 pp.

  
\bibitem{LeeLee:12} J. H. Lee, S. Lee, 
\textit{Inequivalent handlebody-knots with homeomorphic complements},
Algebr. Geom. Topol. {\bf 12} (2012), 1059–-1079. 
 
}

  
\bibitem{Mag:39} W. Magnus:
\textit{\"Uber freie Faktorgruppen und freie Untergruppen gegebener Gruppen}, 
Monatsh. Math. {\bf 47} (1939), 307--313. 
 
%

\nada{
\bibitem{Man:00}
V. Manturov,
\textit{Knot Theory},
Chapman \& Hall/CRC, 2000.

 
%
 
 
\bibitem{Maz:61} B. C. Mazur,
\textit{On embeddings of spheres},
Acta Math. {\bf 105} (1961), 1–-17.

 
\bibitem{Miz:13}, A. Mizusawa,
\textit{Linking numbers for handlebody-links},
Proc. Japan Acad., {\bf 89}, Ser. A (2013), 60--62. 


 
 

\bibitem{Moi:52} E. E. Moise,
\textit{Affine structures in $3$-manifolds: II. Positional Properties of 2-Spheres},
Ann. of Math. {\bf 55} 
(1952), 172--176
 
   
 
 
\bibitem{Mun:59} J. Munkres
\textit{Obstructions to the smoothing of piecewise-differentiable homeomorphisms},
Bull. Amer. Math. Soc. {\bf 65} (5) (1959), 332--334.




\bibitem{Moi:77} E. E. Moise,
\textit{Geometric Topology in Dimensions $2$ and $3$}, Springer-Verlag, Grad. Texts in Math. {\bf 47} (1977).
}

\bibitem{Mor:07} H. Moriuchi:
\textit{An enumeration of theta-curves with up to seven crossings}
J. Knot Theory Ramifications {\bf 18} (2) (2009) 67--197.  

\nada{
\bibitem{Mor:07b} H. Moriuchi,
\textit{A table of handcuff graphs with up to seven crossings}
OCAMI Studies Vol 1. Knot Theory for Scientific objects (2007) 179--300.

\bibitem{Mor:09} H. Moriuchi,
\textit{A table of $\theta$-curves and handcuff graphs with up to seven crossings} 
Adv. Stud. Pure Math. Noncommutativity and Singularities: Proceedings of French--Japanese symposia held at IHES in 2006, J.-P. Bourguignon, M. Kotani, Y. Maeda and N. Tose, eds. (Tokyo: Mathematical Society of Japan, 2009), 281--290. 
  
\bibitem{Mot:90} M. Motto,
\textit{Inequivalent genus two handlebodies in $S^{3}$ with homeomorphic complements}, Topology Appl. {\bf 36} (3) (1990), 283--290

 
 
\bibitem{Neu:43} B. H. Neumann,
\textit{On the number of generators of a free product}, 
Journal of the London Mathematical Society {\bf 18} (1943), 12--20. 
 
\bibitem{Och:91} M. Ochiai,
\textit{Heegaard diagrams of 3-manifolds},
Trans. Amer. Math. Soc. {\bf 328} (2) (1991), 863--879.
 
}
\bibitem{appcontour} M. Paolini:
Appcontour. Computer software. Vers. 2.5.3.
Apparent contour. (2018) $<$\url{http://appcontour.sourceforge.net/}$>$.

\nada{  
\bibitem{RamGovKau:94} P. Ramadevi, T. R. Govindarajan, R. K. Kaul,
\textit{Chirality of knots $9_{42}$ and $10_{71}$ and Chern-Simons theory},
Mod. Phys. Lett. A9 (1994), 3205--3218.

\bibitem{Ril:71} R. Riley,
\textit{Homomorphisms of knot groups on finite groups},
Math. Comp. {\bf 25} (1971), 603--619
  
  
\bibitem{Rol:03} D. Rolfsen,
\textit{Knots and Links}, AMS Chelsea Publishing, vol.364, 2003.

 
 
\bibitem{Sch:27} O. Schreier 
\textit{Die Untergruppen der freien Gruppe},
Abh. Math. Semin. Uni. Hamburg {\bf 5} (1927), 161--183. 





\bibitem{Sma:59} S. Smale
\textit{Diffeomorphisms of the 2-Sphere},
Proc. Amer. Math. Soc.
{\bf 10} (4) (1959) 621--626. 

Proc. Amer. Math. Soc.


\bibitem{SteThu:97} W. P. Thurston,
\textit{Three-dimensional Geometry and Topology, Volume 1}, Princeton University Press, (1997).
 
\bibitem{Sti:87} J. Stillwell,
\textit{In Papers on Group Theory and Topology, Appendix: The Dehn-Nielsen Theorem}, Springer-Verlag, New York Inc., (1987).

\bibitem{Sta:65a} J. R. Stallings, 
\textit{A topological proof of Grushko's theorem on free products } 
Math. Z. {\bf 90} (1965), 1--8. 
  
} 
 
\bibitem{Sta:65b} J. Stallings:
\textit{Homology and central series of groups},
J. Algebra {\bf 2} (1965), 170--181. 
 
\bibitem{Stam:67} U. Stammbach: 
\textit{Ein neuer Beweis eines Satzes von Magnus},
Proc. Camb. Phil. Soc. {\bf 63} (1967), 929--930. 


\bibitem{Suz:70} S. Suzuki:
\textit{On linear graphs in 3–sphere},
Osaka J. Math. {\bf 7} (1970), 375--396.

\bibitem{Suz:75} S. Suzuki:
\textit{On surfaces in 3-sphere: prime decompositions},
Hokkaido Math. J. {\bf 4} (1975), 179--195. 


\nada{
\bibitem{Swa:70} G. A. Swarup,
\textit{Some properties of 3-manifolds with boundary},
Quart. J. Math. Oxford {\bf 21} (1970), 1--23.
}
\bibitem{Tsu:70} Y. Tsukui:
\textit{On surfaces in 3-space}, 
Yokohama Math. J. {\bf 18} (1970), 93--104.  


\bibitem{Tsu:75} Y. Tsukui:
\textit{On a prime surface of genus $2$ and homeomorphic splitting of $3$-sphere},
The Yokohama Math. J. {\bf 23} (1975), 63--75.

\nada{
\bibitem{Wal:67}F. Waldhausen,
\textit{Gruppen mit Zentrum und 3-dimensionale Mannigfaltigkeiten},
Topology {\bf 6} (1967), 505--517.


\bibitem{Wal:68i} F. Waldhausen,
\textit{Heegaard-Zerlegungen der 3-Sph\"are}, 
Topology {\bf 7} (1968), 195--203. 

 
\bibitem{Wal:68ii} F. Waldhausen,
\textit{On irreducible 3-manifolds which are sufficiently large},
Ann. of Math. {\bf 87} (1968), 56--88.
 
  

\bibitem{Whi:87} W. Whitten, 
\textit{Knot complements and groups},
Topology {\bf 26}, Issue 1, (1987), 41--44.
 
 
}

 
  
 
 

 
  

%
\end{thebibliography}
\end{document}